\documentclass[a4paper,12pt,reqno]{article}
\usepackage[latin1]{inputenc}
\usepackage[T1]{fontenc}
\usepackage{lmodern}
\usepackage[english]{babel}
\usepackage{amsmath}
\usepackage{amssymb}
\usepackage{amsthm}
\usepackage{bm}
\usepackage{MnSymbol}
\usepackage[pagestyles]{titlesec}\usepackage[pdfpagemode=UseNone,pdfstartview=FitH]{hyperref}
\renewpagestyle{plain}{\setfoot[\thepage][][]{}{}{\thepage}}
\theoremstyle{plain}
\newtheorem{theorem}{Theorem}[section]

\newtheorem{lemma}{Lemma}[section]
\newtheorem{proposition}{Proposition}[section]
\newtheorem{conjecture}{Conjecture}[section]
\theoremstyle{definition}
\newtheorem{definition}{Definition}[section]
\newtheorem{example}{Example}[section]
\newtheorem{remark}{Remark}[section]
\numberwithin{equation}{section}
\allowdisplaybreaks[1]
  
\newcommand*{\Zset}{\mathbb{Z}}  
  
\newcommand*{\Rset}{\mathbb{R}}  
\newcommand*{\Cset}{\mathbb{C}}
\newcommand*{\mx}[1]{\bm{#1}}

\newcommand*{\ltriplevert}{\left|\mkern-2mu\left|\mkern-2mu\left|}
\newcommand*{\rtriplevert}{\right|\mkern-2mu\right|\mkern-2mu\right|}
\newcommand*{\matrixnorm}[1]{\ltriplevert #1\rtriplevert}
\newcommand*{\abs}[1]{\lvert#1\rvert}
\begin{document}
\title{\textbf{On the eigenvalues of combined meet and join matrices}}
\author{}
\date{\today}
\maketitle
\begin{center}
\textsc{Mika Mattila}\\
School of Information Sciences\\
FI-33014 University of Tampere, Finland\\[5mm]
\end{center}
\begin{abstract}
In this article we give bounds for the eigenvalues of a matrix, which can be seen as a common generalization of meet and join matrices and therefore also as a generalization of both GCD and LCM matrices. Although there are some results concerning the factorizations, the determinant and the inverse of this so-called combined meet and join matrix, the eigenvalues of this matrix have not been studied earlier. Finally we also give a nontrivial lower bound for a certain constant $c_n$, which is needed in calculating the above-mentioned eigenvalue bounds in practice. So far there are no such lower bounds to be found in the literature.

\end{abstract}
\textit{Tel.:}\ +358\ 50\ 318\ 5881\\
\emph{E-mail address:}\ mika.mattila@uta.fi

\newpage

\section{Introduction}

The concept of a meet matrix was first defined by Indian mathematician Bhat  in 1991 \cite{bhat}, whereas join matrices first appeared in a paper by Korkee and Haukkanen in 2003 \cite{haukkanen4}. There are also many other papers about these matrices by Haukkanen and Korkee, see e.g. the references in \cite{MH}. Meet and join matrices were also studied by Hong and Sun in 2004 \cite{hong2}. Both concepts are natural generalizations of GCD and LCM matrices presented by Smith as early as in 1875 \cite{smith}. The definitions are as follows: Assume that $(P,\preceq)$ is a locally finite lattice, $f$
is a real or complex-valued function on $P$ and $S=\{x_1,x_2,\ldots,x_n\}$ is a finite set of distinct elements of $P$ such that
\begin{equation}\label{eq:condition1}
x_i\preceq x_j\Rightarrow i\leq j.
\end{equation}
The $n\times n$ matrix having $f(x_i\wedge x_j)$ as its $ij$ element is the \emph{meet matrix} of the set $S$ with respect to $f$ and is denoted by $(S)_f$. Similarly, the $n\times n$ matrix having $f(x_i\vee x_j)$ as its $ij$ element is the \emph{join matrix} of the set $S$ with respect to $f$ and is denoted by $[S]_f$. When $(P,\preceq)=(\Zset_+,|)$, where $|$ stands for the usual divisor relation of positive integers, the matrices $(S)_f$ and $[S]_f$ are referred to as the GCD and LCM matrices of the set $S$ with respect to $f$. Another simple but important special case of meet and join matrices are MIN and MAX matrices, which are obtained when $(P,\preceq)$ is a chain. The MIN matrix of size $n\times n$ with $\min(i,j)$ as its $ij$ element has been studied by Bhatia \cite{Bh}, for example, and this matrix can easily be seen as a meet matrix by setting $(P,\preceq)=(\Zset_+,\leq)$, $S=\{1,2,\ldots,n\}$ and $f(m)=m$ for all $m\in\Zset_+$.

There are several possible ways to further generalize the concept of meet and/or join matrices. One way to do this is to consider two sets instead of one set $S$ (see \cite{ATH, MH}); another is to replace the function $f$ with $n$ functions $f_1,\ldots,f_n$ (see \cite{MH2}). Korkee \cite{ISmo} defines yet another distinct generalization: a combined meet and join matrix $\mx M_{S,f}^{\alpha, \beta, \gamma, \delta}$. What is special in this generalization is that it yields both meet and join matrices as its special cases, whereas the other generalizations yield only one of the two.

Although the structure, the determinant and the inverse of the matrix $\mx M_{S,f}^{\alpha, \beta, \gamma, \delta}$ were studied by Korkee \cite{ISmo}, there are currently no results concerning the eigenvalues of the general form of this matrix. Our main goal with this paper is to improve this situation. The task, however, is not very easy. Already in the case of more specific GCD and LCM matrices accessing the asymptotic behaviour of the eigenvalues of these matrices requires some rather complicated methods, see e.g. \cite{H08, HLe08, HLo11}. In order to study the eigenvalues of much more general matrix $\mx M_{S,f}^{\alpha, \beta, \gamma, \delta}$ we need to use at least as complicated methods at a more abstract level.

When studying a generalization of a matrix class, it is sometimes possible to extend some methods and results to consider the larger class (at least by making suitable assumptions). When Hong and Loewy obtained a lower bound for the smallest eigenvalue of certain GCD matrices (see \cite[Theorem 4.2]{hong}); soon afterwards Ilmonen et al. \cite{IHM} generalized this result to meet and join matrices. In this article, we show that, under certain circumstances, this method can be extended for the much more general matrix $\mx M_{S,f}^{\alpha, \beta, \gamma, \delta}$. The same goes for another method developed by Ilmonen et al.: see \cite[Theorem 4.1 and Theorem 6.1]{IHM}. This is done in Sections 3 and 4.

In Section 5 we turn our attention to the special constants $c_n$ originally defined by Hong and Loewy. Currently, no lower bounds are known for this constant for general $n$, which means that some of the results in \cite{hong} and in \cite{IHM} cannot be applied in practice at all. It turns out that we were able to contribute something to this topic as well, in this article. 
 
\section{Preliminaries}

Throughout this paper, $(P,\preceq)$ is a locally finite lattice, $f$
is either a real or a complex-valued function on $P$ and $S=\{x_1,x_2,\ldots,x_n\}$ is a finite set of distinct elements of $P$ such that
\begin{equation}\label{eq:condition2}
x_i\preceq x_j\Rightarrow i\leq j.
\end{equation}
In Proposition \ref{esityslause1} and in Theorem \ref{omatulos3} we also assume that $P$ has $\hat{0}$ as its smallest element, and in Proposition \ref{aat-lause2} and in Theorem \ref{omatulos4} $P$ is supposed to have the largest element $\hat{1}$. These assumptions may, however, sound more restricting than they in fact are. If $P$ does not have the smallest or the largest element, we may always restrict ourselves to the finite interval
\[
\boldsymbol{\lsem}\wedge S, \vee S\boldsymbol{\rsem}=\{z\in P\,\big|\,\wedge\hspace{-0.5mm} S\preceq z\preceq\vee S\boldsymbol\},
\]
see e.g. \cite[Section 2]{MH}. Furthermore, the set $S$ is said to be \emph{meet closed} if $x_i\wedge x_j\in S$ for all $x_i,x_j\in S$, or in other words, if the structure $(S,\preceq)$ is a meet semilattice. Similarly the set $S$ is \emph{join closed} if $x_i\vee x_j\in S$ for all $x_i,x_j\in S$ (i.e. $(S,\preceq)$ is a join semilattice).

Next let us recall the definition of a combined meet and join matrix by Korkee \cite{ISmo}:
\begin{definition}[\cite{ISmo}, p. 76]
Let $\mx M_{S,f}^{\alpha, \beta, \gamma, \delta}=[m_{ij}]\in\Cset^{n\times n}$ with
\[
m_{ij}=\frac{f(x_i\wedge x_j)^\alpha f(x_i\vee x_j)^\beta}{f(x_i)^\gamma f(x_j)^\delta},
\]
where $\alpha,\beta,\gamma,\delta$ are real numbers such that the matrix $\mx M_{S,f}^{\alpha, \beta, \gamma, \delta}$ exists.
\end{definition}

In order for the matrix $\mx M_{S,f}^{\alpha, \beta, \gamma, \delta}$ to exist whenever possible, we need to make the agreement that $0^0=1$, but even this does not entirely solve the problem.  The following remark provides detailed criteria for the existence of the matrix $\mx M_{S,f}^{\alpha, \beta, \gamma, \delta}$.

\begin{remark}\label{olemassaolo}
The matrix $\mx M_{S,f}^{\alpha, \beta, \gamma, \delta}$ exists if and only if the following conditions are satisfied:
\begin{enumerate}
\item If $f(x)=0$ for some $x\in S$, then $\gamma=\delta=0$,
\item If $f(x_i\wedge x_j)=0$ for some $x_i,x_j\in S$, then $\alpha\geq0$,
\item If $f(x_i\vee x_j)=0$ for some $x_i,x_j\in S$, then $\beta\geq0$.
\end{enumerate}
\end{remark}

By setting $\alpha=1$ and $\beta=\gamma=\delta=0$ we obtain $\mx M_{S,f}^{1, 0, 0, 0}=(S)_f$. On the other hand, if $\beta=1$ and $\alpha=\gamma=\delta=0$, then $\mx M_{S,f}^{0, 1, 0, 0}=[S]_f$. Thus the name \emph{combined meet and join matrix} is well justified.

Next we present the two factorization theorems for the matrix $\mx M_{S,f}^{\alpha, \beta, \gamma, \delta}$ given by Korkee \cite{ISmo}. The former makes use of the meet matrix $(S)_f$, whereas the latter uses the join matrix $[S]_f$. Here $\mx A\circ \mx B$ denotes the Hadamard product of the matrices $\mx A$ and $\mx B$ and $f^\alpha$ is simply the usual power of the function $f$ with $f^\alpha(x)=[f(x)]^\alpha$ for all $x\in P$.

\begin{proposition}[\cite{ISmo}, Theorem 3.1 (meet-oriented structure theorem)]\label{ISmotheorem1}
Let $\alpha, \beta, \gamma, \delta$ be real numbers such that the matrix $\mx M_{S,f}^{\alpha, \beta, \gamma, \delta}$ exists. Then
\[
\mx M_{S,f}^{\alpha, \beta, \gamma, \delta}=\mx F^{\beta-\gamma}((S)_{f^{\alpha-\beta}}\circ \mx G)\mx F^{\beta-\delta},
\]
where $\mx F=\mathrm{diag}(f(x_1),f(x_2),\ldots,f(x_n))$ and
\[
(\mx G)_{ij}=
\left\{ \begin{array}{ll}
1 & \textrm{if}\ x_i\preceq x_j\ \text{or}\ x_j\preceq x_i,\\
\frac{f^\beta(x_i\wedge x_j)f^\beta(x_i\vee x_j)}{f^\beta(x_i)f^\beta(x_j)} & \textrm{otherwise}.\\
\end{array}\right.
\]
\end{proposition}

\begin{proposition}[\cite{ISmo}, Theorem 3.2 (join-oriented structure theorem)]\label{ISmotheorem2}
Let $\alpha, \beta, \gamma, \delta$ be such real numbers that the matrix $\mx M_{S,f}^{\alpha, \beta, \gamma, \delta}$ exists. Then
\[
\mx M_{S,f}^{\alpha, \beta, \gamma, \delta}=\mx F^{\alpha-\gamma}([S]_{f^{\beta-\alpha}}\circ \mx G)\mx F^{\alpha-\delta},
\]
where $\mx F=\mathrm{diag}(f(x_1),f(x_2),\ldots,f(x_n))$ and
\[
(\mx G)_{ij}=
\left\{ \begin{array}{ll}
1 & \textrm{if}\ x_i\preceq x_j\ \text{or}\ x_j\preceq x_i,\\
\frac{f^\alpha(x_i\wedge x_j)f^\alpha(x_i\vee x_j)}{f^\alpha(x_i)f^\alpha(x_j)} & \textrm{otherwise}.\\
\end{array}\right.
\]
\end{proposition}
After applying the previous two propositions, we also need to be able to factorize the usual meet and join matrices. The following four propositions help us with this. In order to shorten our notations, we introduce two so called \emph{restricted incidence functions} as well as a convolution operation for incidence functions. The function $f_d$ is defined on $\{\hat{0}\times P\}$, $f_u$ on $P\times \{\hat{1}\}$ and
\[
f_d(\hat{0},z)=f(z)=f_u(z,\hat{1})
\]
for all $z\in P$. The convolution  of incidence functions $f$ and $g$ is the incidence function $f\ast g$ for which
\[
(f\ast g)(x,y)=\sum_{x\preceq z\preceq y}f(x,z)g(z,y)
\]
for all $x,y\in P$. Another thing that we need is the Möbius function $\mu_P$ of the poset $P$. The function $\mu_P$ is usually defined as being the inverse of certain incidence function $\zeta$ with respect to the convolution (see \cite[p. 296]{McCarthy} and \cite[p. 141]{Ai}), but it may be more convenient to calculate its values recursively by using the formula
\[
\mu_P(x,y)=\left\{ \begin{array}{lll}
1 & \textrm{if}\ x=y,\\
-\sum\limits_{x\prec z\preceq y}\mu_P(z,y)=-\sum\limits_{x\preceq z\prec y}\mu_P(x,z) & \textrm{if $x\prec y$},\\
0 & \textrm{otherwise,}
\end{array}\right.
\]
see e.g. \cite[Proposition 4.6]{Ai}. This enables us to write briefly by using the convolution $\ast$ as
\[
\sum_{\hat{0}\preceq z\preceq w} f(z)\mu_P(z,w)=(f_d\ast\mu_P)(w)\text{ and }\sum_{w\preceq z\preceq \hat{1}} f(z)\mu_P(w,z)=(\mu_P \ast f_u)(w).
\]
Before going into the factorization theorems we need to deploy two concepts from lattice theory. First, let us assume that $\hat{0}$ is the smallest element of the lattice $(P,\preceq)$. The \emph{order ideal generated by the set $S$} is the set
\[
\{w\in P\ \big|\ \hat{0}\preceq w\preceq x_i\text{ for some }x_i\in S\}=\bigcup_{i=1}^n\lsem\hat{0},x_i\rsem
\]
and it is denoted by $\downarrow\hspace{-1mm} S$. Similarly, if we assume that $\hat{1}$ is the largest element of the lattice $(P,\preceq)$, we may define the \emph{order filter generated by the set $S$} as being the set
\[
\{w\in P\ \big|\ x_i\preceq w\preceq \hat{1}\text{ for some }x_i\in S\}=\bigcup_{i=1}^n\lsem x_i,\hat{1}\rsem,
\]
for which we use the notation $\uparrow\hspace{-1mm} S$.

\begin{proposition}[\cite{haukkanen2}, Lemma 3.2]\label{esityslause1}
Let $\downarrow\hspace{-1mm} S=\{w_1,w_2,\ldots,w_m\}$ and $\mx A=(a_{ij})$ be the $n\times m$ matrix with
\[
a_{ij}=\left\{ \begin{array}{ll}
\sqrt{(f_d*\mu_P)(\hat{0},w_j)} & \textrm{if}\ w_j\preceq x_i,\\
0 & \textrm{otherwise}\\
\end{array}\right.
\]
Then $(S)_f=\mx A\mx A^T$.
\end{proposition}

\begin{proposition}[\cite{haukkanen4}, Lemma 4.2]\label{aat-lause2}
Let $\uparrow\hspace{-1mm} S=\{w_1,w_2,\ldots,w_m\}$ $\mx A=(a_{ij})$ be the $n\times m$ matrix with
\[
a_{ij}=\left\{ \begin{array}{ll}
\sqrt{(\mu_P*f_u)(w_j,\hat{1})} & \textrm{if}\ x_i\preceq w_j,\\
0 & \textrm{otherwise.}\\
\end{array}\right.
\]
Then $[S]_f=\mx A\mx A^T$.
\end{proposition}

\begin{proposition}[\cite{bhat}, Theorem 12]\label{edet-lause1}
Let $S$ be a meet closed set and let $\mx E$ and $\mx D=\mathrm{diag}(d_1,d_2,\ldots,d_n)$ be the $n\times n$ matrices with
\[
e_{ij}=\left\{ \begin{array}{ll}
1 & \textrm{if}\ x_j\preceq x_i,\\
0 & \textrm{otherwise}\\
\end{array}\right.
\]
and
\[
d_i=\sum_{\substack{z\preceq x_i \\ z\npreceq x_j\ \mathrm{for}\ j<i}}(f_d*\mu_P)(\hat{0},z).
\]
Then $(S)_f=\mx E\mx D\mx E^T$.
\end{proposition}

\begin{proposition}[\cite{IHM}, Proposition 2.5]\label{edet-lause2}
Let $S$ be a join closed set and let $\mx E$ and $\mx D=\mathrm{diag}(d_1,d_2,\ldots,d_n)$ be the $n\times n$ matrices with
\[
e_{ij}=\left\{ \begin{array}{ll}
1 & \textrm{if}\ x_j\preceq x_i,\\
0 & \textrm{otherwise}\\
\end{array}\right.
\]
and
\[
d_i=\sum_{\substack{x_i\preceq z \\ x_j\npreceq z\ \mathrm{for}\ i<j}}(\mu_P*f_u)(z,\hat{1}).
\]
Then $[S]_f=\mx E^T\mx D\mx E$.
\end{proposition}

Before we can use these factorizations to estimate the eigenvalues of the matrix $\mx M_{S,f}^{\alpha, \beta, \gamma, \delta}$, we also need the following lemma.

\begin{lemma}\label{Hadamardin tulo}
Let $\mx A=[a_{ij}],\mx B=[b_{ij}],\mx C=[c_{ij}],\mx D=[d_{ij}]\in\Cset^{n\times n}$, where $\mx C$ and $\mx D$ are diagonal matrices. Then
\[
\mx C(\mx A\circ\mx B)\mx D=\mx B\circ(\mx C\mx A\mx D).
\]
\end{lemma}
\begin{proof}
Since
\begin{align*}
&(\mx C(\mx A\circ\mx B)\mx D)_{ij}=\sum_{k=1}^nc_{ik}((\mx A\circ\mx B)\mx D)_{kj}=\sum_{k=1}^nc_{ik}\left(\sum_{l=1}^n(\mx A\circ\mx B)_{kl}d_{lj}\right)\\
&=\sum_{k=1}^nc_{ik}\left(\sum_{l=1}^na_{kl}b_{kl}d_{lj}\right)=\sum_{k=1}^n\sum_{l=1}^nc_{ik}a_{kl}b_{kl}\cdot\underbrace{d_{lj}}_{\substack{=0\\ \text{when}\ l\neq j}}\\
&=\sum_{k=1}^n\underbrace{c_{ik}}_{\substack{=0\\ \text{when}\ i\neq k}}\cdot a_{kj}b_{kj}d_{jj}=c_{ii}a_{ij}b_{ij}d_{jj}\\
&=b_{ij}((c_{ii}a_{ij})d_{jj})=b_{ij}\left(\left(\sum_{k=1}^nc_{ik}a_{kj}\right)d_{jj}\right)
=b_{ij}\left((\mx C\mx A)_{ij}d_{jj}\right)\\&=b_{ij}\left(\sum_{k=1}^n(\mx C\mx A)_{ik}d_{kj}\right)
=b_{ij}((\mx C\mx A)\mx D)_{ij}=(\mx B\circ(\mx C\mx A\mx D))_{ij},
\end{align*}
the claim follows.
\end{proof}

In the following two sections we need to assume that our function $f$ is \emph{semimultiplicative,} which means that
\begin{equation*}\label{eq:semi}
f(x)f(y)=f(x\land y)f(x\lor y)
\end{equation*}
for all $x,y\in P$. We also adopt one constant $c_n$ from Hong and Loewy \cite{hong} and another $C_n$ from Ilmonen et al. \cite{IHM}. Let $K(n)$ denote the set of all $n\times n$ lower triangular $0,1$ matrices with each main diagonal element equal to $1$. Now for every positive integer $n$ we define
\[
c_n=\min\{\lambda\,\big|\,\mx X \in K(n)\text{ and $\lambda$ is the smallest eigenvalue of }\mx X\mx X^T\}
\]
and
\[
C_n=\max\{\lambda\,\big|\,\mx X \in K(n)\text{ and $\lambda$ is the largest eigenvalue of }\mx X\mx X^T\}.
\]

Finally, we introduce some old and new notations concerning matrix analysis. We denote that $\mx J$ is the $n\times n$ matrix with all its elements equal to $1$ (i.e. $\mx J$ is the identity element under the Hadamard product of complex $n\times n$ matrices). If $\mx A$ and $\mx B$ are real matrices, the notation $\mx A\leqslant \mx B$ is used for the componentwise inequality (that is, $a_{ij}\leq b_{ij}$ for all $i,j=1,\ldots,n$). In this paper, $\abs{\mx A}$ does not stand for the determinant of $\mx A$, but for the $n\times n$ matrix, with $\abs{a_{ij}}$ as its $ij$ element. The Frobenius and spectral norms of a given matrix $\mx A$ are denoted by $\matrixnorm{\mx A}_F$ and $\matrixnorm{\mx A}_S$ respectively. 
As usual, the spectral radius $\rho(\mx A)$ of a matrix $\mx A$ is defined to be the maximum of the absolute values of the eigenvalues of $\mx A$. For the purposes of this paper, it is convenient to deploy similar notation for the smallest absolute value of the eigenvalues of the matrix $\mx A$. We denote
\[
\kappa(\mx A)=\min\{\abs{\lambda}\,\big|\,\lambda\text{ is an eigenvalue of }\mx A\}.
\]
For example, if $\mx A$ is invertible and Hermitean, then
\[
\rho(\mx A^{-1})=\matrixnorm{\mx A^{-1}}_S=\frac{1}{\kappa(\mx A)}.
\]
\section{Lower bound for the smallest eigenvalue of a positive definite combined meet and join matrix}

Under suitable circumstances the matrix $\mx M_{S,f}^{\alpha, \beta, \gamma, \delta}$ becomes positive definite and it is thus possible to find a real lower bound for its smallest eigenvalue by making use of the structure theorems presented earlier. 

\begin{theorem}\label{omatulos3}
Let $\alpha,\beta,\gamma,\delta$ be real numbers such that $\gamma=\delta$ and the matrix $\mx M_{S,f}^{\alpha, \beta, \gamma, \gamma}$ exists. Let $f:P\to\Rset\backslash\{0\}$ be a semimultiplicative function and $\downarrow\hspace{-1mm} S=\{w_1,w_2,\ldots,w_m\}$. If $(f_d^{\alpha-\beta}*\mu_P)(\hat{0},w_i)>0$ for all $w_i\in\downarrow\hspace{-1mm} S$, then
\[
\kappa(\mx M_{S,f}^{\alpha, \beta, \gamma, \gamma})\geq c_n\cdot \min_{1\leq i\leq n}(f_d^{\alpha-\beta}*\mu_P)(\hat{0},x_i)\cdot\min_{1\leq i\leq n}[f^2(x_i)]^{\beta-\gamma}.
\]
\end{theorem}
\begin{proof}
Let $\mx A=(a_{ij})$ be the $n\times m$ matrix with
\[
a_{ij}=\left\{ \begin{array}{ll}
\sqrt{(f_d^{\alpha-\beta}*\mu_P)(\hat{0},w_j)} & \textrm{if}\ w_j\preceq x_i,\\
0 & \textrm{otherwise.}\\
\end{array}\right.
\]
and $\mx F=\text{diag}(f(x_1),\ldots,f(x_n)).$ With Proposition \ref{esityslause1} we have $(S)_{f^{\alpha-\beta}}=\mx A\mx A^T$. We may assume that $w_i=x_i$ for all $i\in\{1,2,\ldots,n\}$, since rearranging the order of the elements of the set $\downarrow\hspace{-1mm} S$ corresponds to permuting some of the rows and respective columns of $(S)_f$, which does not affect to the eigenvalues.

The matrix $\mx A$ can now be divided into blocks
\[
\mx A=[\mx B\ |\ \mx C],
\]
where $\mx B$ is an $n\times n$ matrix and $\mx C$ is of size $n\times (m-n)$. Since $f$ is a semimultiplicative function, every element of the matrix $\mx G$ defined in Proposition \ref{ISmotheorem1} is equal to $1$. By applying this proposition we obtain
\begin{align}\label{eq:hajotus}
\mx M_{S,f}^{\alpha, \beta, \gamma, \gamma}&=\mx F^{\beta-\gamma}((S)_{f^{\alpha-\beta}}\circ\mx G)\mx F^{\beta-\gamma}
=\mx F^{\beta-\gamma}((S)_{f^{\alpha-\beta}}\circ\mx J)\mx F^{\beta-\gamma}\notag\\&=\mx F^{\beta-\gamma}(S)_{f^{\alpha-\beta}}\mx F^{\beta-\gamma}=\mx F^{\beta-\gamma}(\mx A\mx A^T)\mx F^{\beta-\gamma}\notag\\&=\mx F^{\beta-\gamma}\left([\mx B\ |\ \mx C][\mx B\ |\ \mx C]^T\right)\mx F^{\beta-\gamma}=\mx F^{\beta-\gamma}\left([\mx B\ |\ \mx C]\left[ \begin{array}{c}
\mx B^T \\
\hline
\mx C^T 
\end{array} \right]\right)\mx F^{\beta-\gamma}\notag\\&=\mx F^{\beta-\gamma}(\mx B\mx B^T+\mx C\mx C^T)\mx F^{\beta-\gamma}=\mx F^{\beta-\gamma}\mx B\mx B^T\mx F^{\beta-\gamma}+\mx F^{\beta-\gamma}\mx C\mx C^T\mx F^{\beta-\gamma}\notag\\&=(\mx F^{\beta-\gamma}\mx B)(\mx F^{\beta-\gamma}\mx B)^T+(\mx F^{\beta-\gamma}\mx C)(\mx F^{\beta-\gamma}\mx C)^T.
\end{align}
Here the matrix $(\mx F^{\beta-\gamma}\mx C)(\mx F^{\beta-\gamma}\mx C)^T$ is clearly positive semidefinite, and thus \cite[Corollary 4.3.12]{matrixanalysis} implies that
\[
\kappa(\mx M_{S,f}^{\alpha, \beta, \gamma, \gamma})\geq\kappa((\mx F^{\beta-\gamma}\mx B)(\mx F^{\beta-\gamma}\mx B)^T).
\]
Let us then consider the $n\times n$ matrix $\mx B=(b_{ij})$ with
\[
b_{ij}=\left\{ \begin{array}{ll}
\sqrt{(f_d^{\alpha-\beta}*\mu_P)(\hat{0},x_j)} & \textrm{if}\ x_j\preceq x_i,\\
0 & \textrm{otherwise.}\\
\end{array}\right.
\]
Let $\mx E$ be the matrix defined in Proposition \ref{edet-lause1} and $\mx D=\text{diag}(d_1,\ldots,d_n)$, where
\[d_i=\sqrt{(f_d^{\alpha-\beta}*\mu_P)(\hat{0},x_i)}.\]
The matrix $\mx B$ can now be written as
\[\mx B=\mx E\mx D.\]
In addition,
\begin{align*}
\det(\mx F^{\beta-\gamma}\mx B)&=\det(\mx F^{\beta-\gamma})\det(\mx E)\det(\mx D)\\&=\prod_{i=1}^n [f(x_i)]^{\beta-\gamma}\cdot 1\cdot\prod_{i=1}^n\sqrt{(f_d^{\alpha-\beta}*\mu_P)(\hat{0},x_i)}\neq0,
\end{align*}
which means that the matrix $\mx F^{\beta-\gamma}\mx B$ is invertible.
Therefore the greatest eigenvalue of the matrix \[[(\mx F^{\beta-\gamma}\mx B)(\mx F^{\beta-\gamma}\mx B)^T]^{-1}=((\mx F^{\beta-\gamma}\mx B)^{-1})^T(\mx F^{\beta-\gamma}\mx B)^{-1}\] is equal to \[\rho([(\mx F^{\beta-\gamma}\mx B)(\mx F^{\beta-\gamma}\mx B)^T]^{-1})=\matrixnorm{[(\mx F^{\beta-\gamma}\mx B)(\mx F^{\beta-\gamma}\mx B)^T]^{-1}}_S.\]
Thus
\begin{align*}
\kappa((\mx F^{\beta-\gamma}\mx B)(\mx F^{\beta-\gamma}\mx B)^T)&=\frac{1}{\rho([(\mx F^{\beta-\gamma}\mx B)(\mx F^{\beta-\gamma}\mx B)^T]^{-1})}\\&=\frac{1}{\matrixnorm{[(\mx F^{\beta-\gamma}\mx B)(\mx F^{\beta-\gamma}\mx B)^T]^{-1}}_S}.
\end{align*}
The assumption about the positiveness implies that
\begin{align*}
\matrixnorm{(\mx D^2)^{-1}}_S&=\matrixnorm{\text{diag}(\frac{1}{(f_d^{\alpha-\beta}*\mu_P)(\hat{0},x_1)},\ldots,\frac{1}{(f_d*\mu_P)(\hat{0},x_n)})}_S\\
&=\max_{1\leq i\leq n}\frac{1}{(f_d^{\alpha-\beta}*\mu_P)(\hat{0},x_i)}=\frac{1}{\min_{1\leq i\leq n}(f_d^{\alpha-\beta}*\mu_P)(\hat{0},x_i)}.
\end{align*}

Similarly,
\begin{align*}
\matrixnorm{(\mx F^{2(\beta-\gamma)})^{-1}}_S&=\matrixnorm{\text{diag}(\frac{1}{[f(x_1)]^{2(\beta-\gamma)}},\ldots,\frac{1}{[f(x_n)]^{2(\beta-\gamma)}})}_S\\
&=\max_{1\leq i\leq n}\frac{1}{[f(x_i)]^{2(\beta-\gamma)}}=\frac{1}{\min_{1\leq i\leq n}[f(x_i)]^{2(\beta-\gamma)}}.
\end{align*}

Applying the submultiplicativity of the spectral norm yields
\begin{align*}
&\matrixnorm{[(\mx F^{\beta-\gamma}\mx B)(\mx F^{\beta-\gamma}\mx B)^T]^{-1}}_S\\&=\matrixnorm{(\mx F^{\beta-\gamma}\mx E\mx D\mx D^T\mx E^T(\mx F^{\beta-\gamma})^T)^{-1}}_S\\&=\matrixnorm{(\mx F^{\beta-\gamma}\mx E\mx D^2\mx E^T\mx F^{\beta-\gamma})^{-1}}_S\\&=\matrixnorm{(\mx F^{\beta-\gamma})^{-1}(\mx E^T)^{-1}(\mx D^2)^{-1}\mx E^{-1}(\mx F^{\beta-\gamma})^{-1}}_S\\&\leq \matrixnorm{(\mx F^{\beta-\gamma})^{-1}}_S\cdot\matrixnorm{(\mx E^T)^{-1}}_S\cdot\matrixnorm{(\mx D^2)^{-1}}_S\cdot\matrixnorm{\mx E^{-1}}_S\cdot\matrixnorm{(\mx F^{\beta-\gamma})^{-1}}_S\\&=\matrixnorm{(\mx D^2)^{-1}}_S\cdot(\matrixnorm{(\mx E^{-1})^T}_S\cdot\matrixnorm{\mx E^{-1}}_S)\cdot\matrixnorm{(\mx F^{\beta-\gamma})^{-1}}_S^2\\&=\matrixnorm{(\mx D^2)^{-1}}_S\cdot\matrixnorm{(\mx E^T)^{-1}\mx E^{-1}}_S\cdot\matrixnorm{(\mx F^{2(\beta-\gamma)})^{-1}}_S\\&=\matrixnorm{(\mx D^2)^{-1}}_S\cdot\matrixnorm{(\mx E\mx E^T)^{-1}}_S\cdot\matrixnorm{(\mx F^{2(\beta-\gamma)})^{-1}}_S.
\end{align*}
Since clearly $\mx E\in K(n)$, we must have $\kappa(\mx E\mx E^T)\geq c_n$. Thus
\[
\matrixnorm{(\mx E\mx E^T)^{-1}}_S=\rho((\mx E\mx E^T)^{-1})=\frac{1}{\kappa(\mx E\mx E^T)}\leq \frac{1}{c_n},
\]
and further
\[
\frac{1}{\matrixnorm{(\mx E\mx E^T)^{-1}}_S}\geq c_n.
\]
Now combining all these results yields
\begin{align*}
\kappa(\mx M_{S,f}^{\alpha, \beta, \gamma, \gamma})&\geq\kappa((\mx F^{\beta-\gamma}\mx B)(\mx F^{\beta-\gamma}\mx B)^T)=\frac{1}{\matrixnorm{[(\mx F^{\beta-\gamma}\mx B)(\mx F^{\beta-\gamma}\mx B)^T]^{-1}}_S}\\&\geq \frac{1}{\matrixnorm{(\mx D^2)^{-1}}_S\cdot\matrixnorm{(\mx E\mx E^T)^{-1}}_S\cdot\matrixnorm{(\mx F^{2(\beta-\gamma)})^{-1}}_S}\\&=\frac{1}{\matrixnorm{(\mx E\mx E^T)^{-1}}_S}\cdot\frac{1}{\matrixnorm{(\mx D^2)^{-1}}_S}\cdot\frac{1}{\matrixnorm{(\mx F^{2(\beta-\gamma)})^{-1}}_S}\\&\geq c_n\cdot\min_{1\leq i\leq n}(f_d^{\alpha-\beta}*\mu_P)(\hat{0},x_i)\cdot\min_{1\leq i\leq n}[f(x_i)]^{2(\beta-\gamma)}.
\end{align*}
\end{proof}

\begin{example}
If $\beta=0$, we do not need to assume the semimultiplicativity of $f$ in Theorem \ref{omatulos3}. Also, in this situation, $(P,\preceq)$ does not necessarily have to be a join semilattice, and neither is the assumption about the largest element $\hat{1}$ necessary. If $\beta=0$, we have $\mx G=\mx J$ trivially. And further, if $\gamma=0$, we can also allow $f$ to have zero values and we simply have $\mx F^{\beta-\gamma}=\mx F^0=\mx I$. Thus Theorem 4.1 in \cite{IHM} is a corollary of Theorem \ref{omatulos3} concerning the matrix $\mx M_{S,f}^{1, 0, 0, 0}=(S)_f$.
\end{example}

\begin{example}[\cite{MH3}, Theorem 3.1]\label{Bangkok}
Let $(P,\preceq,\hat{0})=(\Zset_+,|,1)$. Consider the $n\times n$ matrix $\mx A_n^{\alpha,\beta}$ with
\[
(i,j)^\alpha[i,j]^\beta
\]
as its $ij$ element. Suppose that $\alpha>\beta$. Clearly $\gamma=\delta=0$, $S=\{1,\ldots,n\}=\downarrow\hspace{-1.5mm}S$ and $f=N$, where $N(m)=m$ for all $m\in\Zset_+$. The function $N$ is obviously semimultiplicative with nonzero values. In addition, since the set $\{1,\ldots,n\}$ is factor closed, we have $$\mu_P(\hat{0},w_i)=\mu(w_i/1)\text{ for all }1\leq w_i\leq n,$$ where $\mu$ denotes the number-theoretic Möbius function (see \cite[Chapter 7]{McCarthy}). Thus
\[
(f_d^{\alpha-\beta}*\mu_P)(\hat{0},w_i)=(N^{\alpha-\beta}*\mu)(w_i)=J_{\alpha-\beta}(w_i)=w_i^{\alpha-\beta}\prod_{p\,|\,w_i}\left(1-\frac{1}{p^{\alpha-\beta}}\right)>0,
\]
where $J_{\alpha-\beta}$ denotes the generalized Jordan totient function and $\ast$ is the Dirichlet convolution. Furthermore, $\min_{1\leq i\leq n}[f^2(x_i)]^{\beta-\gamma}$ is equal to either $1$ or $n^{2\beta}$. Thus by Theorem \ref{omatulos3} we have
\[
\kappa(\mx A_n^{\alpha,\beta})\geq c_n\cdot \min_{1\leq i\leq n}J_{\alpha-\beta}(i)\cdot\min\{1,n^{2\beta}\}>0.
\]
The difference between this result and Theorem 3.1 of \cite{MH3} is that in \cite{MH3} the constant $c_n$ is replaced with a larger constant $t_n$, which is obtained by calculating the smallest eigenvalue of the matrix $\mx E\mx E^T$, where $\mx E$ is the incidence matrix of the set $\{1,\ldots,n\}$ with respect to the divisor relation (which is not the matrix that yields the constant $c_n$).
\end{example}

Since we assume that $(P,\preceq)$ is not only a semilattice but a lattice, it is also possible to approach the eigenvalues of the matrix $\mx M_{S,f}^{\alpha, \beta, \gamma, \gamma}$ via the join matrix $[S]_f$. In this case we just make use of Propositions \ref{ISmotheorem2} and \ref{aat-lause2} and then proceed as in the proof of Theorem \ref{omatulos3}.

\begin{theorem}\label{omatulos4}
Let $\alpha,\beta,\gamma,\delta$ be real numbers such that $\gamma=\delta$ and the matrix $\mx M_{S,f}^{\alpha, \beta, \gamma, \gamma}$ exists. Let $f:P\to\Rset\backslash\{0\}$ be a semimultiplicative function and $\uparrow\hspace{-1mm} S=\{w_1,w_2,\ldots,w_m\}$. If $(\mu_P*f_u^{\beta-\alpha})(w_i,\hat{1})>0$ for all $w_i\in\uparrow\hspace{-1mm} S$, then
\[
\kappa(\mx M_{S,f}^{\alpha, \beta, \gamma, \gamma})\geq c_n\cdot \min_{1\leq i\leq n}(\mu_P*f_u^{\beta-\alpha})(x_i,\hat{1})\cdot\min_{1\leq i\leq n}[f^2(x_i)]^{\alpha-\gamma}.
\]
\end{theorem}
\begin{proof}
The proof is similar to the proof of Theorem \ref{omatulos3}.
\end{proof}

\begin{example}
Theorem 5.1 in \cite{IHM} follows directly from Theorem \ref{omatulos4}. In this case $\alpha=0$, and therefore $f$ does not need to be semimultiplicative, nor does $(P,\prec)$ need to be a meet semilattice with $\hat{0}$ as the smallest element. If also $\gamma=0$, then trivially $\mx F^{\alpha-\gamma}=\mx I$ and the image of $f$ does not have to be restricted to nonzero values.
\end{example}

Theorems \ref{omatulos3} and \ref{omatulos4} provide two different approaches to the smallest eigenvalue of $\mx M_{S,f}^{\alpha, \beta, \gamma, \gamma}$. It should be noted that the bounds obtained by using these theorems may differ greatly (provided that both theorems are applicable). For example, if the set $\downarrow\hspace{-1.5mm} S$ is much larger than the set $\uparrow\hspace{-1.5mm} S$, then the elements in the difference matrix $(\mx F^{\beta-\gamma}\mx C)(\mx F^{\beta-\gamma}\mx C)^T$ in the proof of Theorem \ref{omatulos3} are likely to be large, which also indicates much poorer lower bound. If the set $\uparrow\hspace{-1mm} S$ is large compared to $\downarrow\hspace{-1mm} S$, then the bound in Theorem \ref{omatulos3} is likely to be much better.

\section{Eigenvalue bound for the combined meet and join matrix of a meet or join closed set}

So far we have been studying the matrix $\mx M_{S,f}^{\alpha, \beta, \gamma, \delta}$ only under the circumstances that it is positive definite. Even if this is not the case, it may still be possible to define regions in the complex plain that contain the eigenvalues. It is then easy to apply these results, for example to a reciprocal matrix with
\[\frac{f(x_i\wedge x_j)}{f(x_i\vee x_j)}\quad \text{or}\quad \frac{f(x_i\vee x_j)}{f(x_i\wedge x_j)}\]
as its $ij$ element. Next we consider the cases when the set $S$ is closed under either operation $\wedge$ or $\vee$. The next theorem is in fact a generalization of Theorem 4.1 in \cite{IHM}.

\begin{theorem}\label{omatulos1}
Let $S$ be a meet closed set, $f$ be a function $P\to\Cset$ and $\alpha,\beta,\gamma,\delta$ be real numbers such that $\gamma=\delta$ and the matrix $\mx M_{S,f}^{\alpha, \beta, \gamma, \gamma}$ exists. If
\begin{equation}
\left|\frac{f(x_i\wedge x_j)f(x_i\vee x_j)}{f(x_i)f(x_j)}\right|^\beta\leq1
\label{eq:ehto2}
\end{equation}
for all $i,j\in\{1,2,\ldots,n\}$, then all the eigenvalues of the matrix $\mx M_{S,f}^{\alpha, \beta, \gamma, \gamma}$ lie in the region
\[
\bigcup_{k=1}^n\hspace{-1mm}\left\{ z\in\Cset\,\Big|\,\abs{z-f(x_k)^{\alpha+\beta-2\gamma}}\leq C_n\cdot\max_{1\leq i\leq n}\abs{f(x_i)}^{2(\beta-\gamma)}\cdot\max_{1\leq i\leq n}\abs{d_i}-\abs{f(x_k)}^{\alpha+\beta-2\gamma} \right\}\hspace{-1mm},
\]
where
\[
d_i=\sum_{\substack{z\preceq x_i \\ z\npreceq x_j\ \mathrm{for}\ j<i}} (f_d^{\alpha-\beta}*\mu_P)(\hat{0},z).
\]
\end{theorem}

\begin{proof}
It follows from condition \eqref{eq:ehto2} that the matrix $\mx G=[g_{ij}]$ defined in Proposition \ref{ISmotheorem1} satisfies
\[
\abs{g_{ij}}=\left|\frac{f(x_i\wedge x_j)f(x_i\vee x_j)}{f(x_i)f(x_j)}\right|^\beta\leq1,
\]
which implies that $\abs{\mx G}\leqslant\mx J$. Let $\mx E$ now be the matrix defined in Proposition \ref{edet-lause1}, $\mx D=\text{diag}(d_1,d_2,\ldots,d_n)$ and 
\[\mx \Lambda=\abs{\mx D}^{\frac{1}{2}}=\text{diag}(\sqrt{\abs{d_1}},\sqrt{\abs{d_2}},\ldots,\sqrt{\abs{d_n}}),\] where
\[
d_i=\sum_{\substack{z\preceq x_i \\ z\npreceq x_j\ \mathrm{for}\ j<i}} (f_d^{\alpha-\beta}*\mu_P)(\hat{0},z).
\]
According to Proposition \ref{edet-lause1}, we have $(S)_{f^{\alpha-\beta}}=\mx E\mx D\mx E^T$. By using the above notations, Proposition \ref{ISmotheorem1} and Lemma \ref{Hadamardin tulo} we obtain
\begin{align*}
\abs{\mx M_{S,f}^{\alpha, \beta, \gamma, \gamma}}=&\abs{\mx F^{\beta-\gamma}((S)_{f^{\alpha-\beta}}\circ \mx G)\mx F^{\beta-\gamma}}=
\abs{(\mx F^{\beta-\gamma}(S)_{f^{\alpha-\beta}}\mx F^{\beta-\gamma})\circ \mx G}\\
=&\abs{\mx F^{\beta-\gamma}(S)_{f^{\alpha-\beta}}\mx F^{\beta-\gamma}}\circ \abs{\mx G}\leqslant
\abs{(\mx F^{\beta-\gamma}(S)_{f^{\alpha-\beta}}\mx F^{\beta-\gamma})}\circ \mx J\\
=&\abs{\mx F^{\beta-\gamma}(S)_{f^{\alpha-\beta}}\mx F^{\beta-\gamma}}
=\abs{\mx F^{\beta-\gamma}}\abs{(S)_{f^{\alpha-\beta}}}\abs{\mx F^{\beta-\gamma}}\\
=&\abs{\mx F}^{\beta-\gamma}\abs{\mx E\mx D\mx E^T}\abs{\mx F}^{\beta-\gamma}
\leqslant\abs{\mx F}^{\beta-\gamma}\mx E\abs{\mx D}\mx E^T\abs{\mx F}^{\beta-\gamma}\\
=&\abs{\mx F}^{\beta-\gamma}\mx E\mx \Lambda\mx \Lambda^T\mx E^T\abs{\mx F}^{\beta-\gamma}
=(\abs{\mx F}^{\beta-\gamma}\mx E\mx \Lambda)(\abs{\mx F}^{\beta-\gamma}\mx E\mx \Lambda)^T.
\end{align*}
With Theorem 8.1.18 in \cite{matrixanalysis} we now have
\[
\rho(\abs{\mx F}^{\beta-\gamma}\abs{(S)_{f^{\alpha-\beta}}}\abs{\mx F}^{\beta-\gamma})\leq\rho(\abs{\mx F}^{\beta-\gamma}\mx E\mx \Lambda\mx \Lambda^T\mx E^T\abs{\mx F}^{\beta-\gamma}).
\]
In addition,
\begin{align}\label{eq:spektraali}
\rho(\abs{\mx F}^{\beta-\gamma}\mx E\mx \Lambda&\mx \Lambda^T\mx E^T\abs{\mx F}^{\beta-\gamma})=\matrixnorm{\abs{\mx F}^{\beta-\gamma}\mx E\mx \Lambda\mx \Lambda^T\mx E^T\abs{\mx F}^{\beta-\gamma}}_S\notag\\
&\leq \matrixnorm{\abs{\mx F}^{\beta-\gamma}}_S\matrixnorm{\mx E}_S\matrixnorm{\mx \Lambda\mx \Lambda^T}_S\matrixnorm{\mx E^T}_S\matrixnorm{\abs{\mx F}^{\beta-\gamma}}_S\notag\\
&=\matrixnorm{\abs{\mx F}^{2(\beta-\gamma)}}_S\matrixnorm{\mx E\mx E^T}_S\matrixnorm{\abs{\mx D}}_S\notag\\
&\leq \max_{1\leq i\leq n}\abs{f(x_i)}^{2(\beta-\gamma)}\cdot C_n\cdot\max_{1\leq i\leq n}\abs{d_i}.
\end{align}
Since $(\mx M_{S,f}^{\alpha, \beta, \gamma, \gamma})_{ii}=f(x_i)^{\alpha+\beta-2\gamma}$ and 
\[(\abs{\mx F}^{\beta-\gamma}\abs{(S)_{f^{\alpha-\beta}}}\abs{\mx F}^{\beta-\gamma})_{ii}=\abs{f(x_i)}^{\alpha+\beta-2\gamma},\]
by using \eqref{eq:spektraali} and by setting $\mx A=\mx M_{S,f}^{\alpha, \beta, \gamma, \delta}$ and $\mx B=\abs{\mx F}^{\beta-\gamma}\abs{(S)_{f^{\alpha-\beta}}}\abs{\mx F}^{\beta-\gamma}$ in \cite[Theorem 8.2.9]{matrixanalysis} it now follows that all the eigenvalues of the matrix $\mx M_{S,f}^{\alpha, \beta, \gamma, \gamma}$ belong to the above-mentioned region.
\end{proof}

\begin{example}
Theorem 4.1 in \cite{IHM} is a consequence of Theorem \ref{omatulos1} We only need to choose $\alpha=1$ and $\beta=\gamma=\delta=0$. Condition \eqref{eq:ehto2} is now trivially satisfied.
\end{example}

\begin{example}\label{resiprookkimatriisi}
Let $S$ be meet closed. Let us consider the reciprocal matrix with $\frac{f(x_i\vee x_j)}{f(x_i\wedge x_j)}$ as its $ij$ element. Thus in this case $\alpha=-1$, $\beta=1$ and $\gamma=\delta=0$. Now if 
\[
\left|\frac{f(x_i\wedge x_j)f(x_i\vee x_j)}{f(x_i)f(x_j)}\right|\leq1
\]
for all $i,j\in\{1,2,\ldots,n\}$, then according to Theorem \ref{omatulos1} all the eigenvalues of the matrix $\mx M_{S,f}^{-1, 1, 0, 0}$ belong to the region
\[
\bigcup_{k=1}^n \left\{ z\in\Cset\ \Big|\ \abs{z-1}\leq C_n\cdot\max_{1\leq i\leq n}\abs{f(x_i)}^2\cdot\max_{1\leq i\leq n}\abs{d_i}-1 \right\} ,
\]
where
\[
d_i=\sum_{\substack{z\preceq x_i \\ z\npreceq x_j\ \textrm{for}\ j<i}} (f_d^{-2}*\mu_P)(\hat{0},z).
\]
Since every set in this union is a disc around $1$, the one with the largest radius also contains all the eigenvalues of the matrix $\mx M_{S,f}^{-1, 1, 0, 0}$.
\end{example}

\begin{example}[\cite{MH3}, Theorem 3.5]
Let $\mx A_n^{\alpha,\beta}$ be the matrix defined in Example \ref{Bangkok}. By applying Theorem \ref{omatulos1} to this matrix, it is easy to see that all the eigenvalues of the matrix $\mx A_n^{\alpha,\beta}$ belong to the region
\[
\bigcup_{k=1}^n\big\{z\in\Cset\ \Big|\ |z-k^{\alpha+\beta}|\leq C_n\cdot\max\{1,n^{2\beta}\}\cdot\max_{1\leq i\leq n}|J_{\alpha-\beta}(i)|-k^{\alpha+\beta}\big\}.
\]
Proceeding now as in the proof of Theorem 3.5 in \cite{MH3} it is possible to show that this union is in fact the real interval $[2\min\{1,n^{\alpha+\beta}\}-H_n,H_n]$, where $H_n=C_n\cdot\max\{1,n^{2\beta}\}\cdot\max_{1\leq i\leq n}|J_{\alpha-\beta}(i)|$.
Also in this case it would be possible to replace the constant $C_n$ with a bit better (i.e. smaller) constant, which can be obtained by using the  exact incidence matrix of the set $\{1,2,\ldots,n\}$.
\end{example}

The next theorem is a result similar to Theorem \ref{omatulos1}, but it is for a join closed set $S$ and is based on Propositions \ref{ISmotheorem2} and \ref{edet-lause2}. The proof is omitted, as it is very similar to the proof of Theorem \ref{omatulos1}.

\begin{theorem}\label{omatulos2}
Let $S$ be a join closed set, $f$ be a function $P\to\Cset$ and $\alpha,\beta,\gamma,\delta$ be real numbers such that $\gamma=\delta$ and the matrix $\mx M_{S,f}^{\alpha, \beta, \gamma, \gamma}$ exists. If
\begin{equation}
\left|\frac{f(x_i\wedge x_j)f(x_i\vee x_j)}{f(x_i)f(x_j)}\right|^\alpha\leq1
\label{eq:ehto4}
\end{equation}
for all $i,j\in\{1,2,\ldots,n\}$, then all the eigenvalues of the matrix $\mx M_{S,f}^{\alpha, \beta, \gamma, \gamma}$ belong to the region
\[
\bigcup_{k=1}^n\hspace{-1mm}\left\{ z\in\Cset\,\Big|\,\abs{z-f(x_k)^{\alpha+\beta-2\gamma}}\leq C_n\cdot\max_{1\leq i\leq n}\abs{f(x_i)}^{2(\alpha-\gamma)}\cdot\max_{1\leq i\leq n}\abs{d_i}-\abs{f(x_k)}^{\alpha+\beta-2\gamma} \right\},
\]
where
\[
d_i=\sum_{\substack{x_i\preceq z \\ x_j\npreceq z\ \mathrm{for}\ i<j}} (\mu_P*f_u^{\beta-\alpha})(z,\hat{1}).
\]
\end{theorem}

\begin{example}
Theorem 6.1 in \cite{IHM} is a consequence of Theorem \ref{omatulos2} and is obtained by setting $\beta=1$ and $\alpha=\gamma=\delta=0$. The condition \eqref{eq:ehto4} holds trivially.
\end{example}

\begin{example}
Let $S$ be join closed. Consider the reciprocal matrix with $\frac{f(x_i\wedge x_j)}{f(x_i\vee x_j)}$ as its $ij$ element. Now $\alpha=1$, $\beta=-1$ and $\gamma=\delta=0$. If also 
\[
\left|\frac{f(x_i\wedge x_j)f(x_i\vee x_j)}{f(x_i)f(x_j)}\right|\leq1
\]
for all $i,j\in\{1,2,\ldots,n\}$, then all the eigenvalues of the matrix $\mx M_{S,f}^{1, -1, 0, 0}$ belong to the region
\[
\bigcup_{k=1}^n \left\{ z\in\Cset\ \Big|\ \abs{z-1}\leq C_n\cdot\max_{1\leq i\leq n}\abs{f(x_i)}^2\cdot\max_{1\leq i\leq n}\abs{d_i}-1 \right\},
\]
where
\[
d_i=\sum_{\substack{x_i\preceq z \\ x_j\npreceq z\ \textrm{for}\ i<j}} (\mu_P*f_u^{-2})(z,\hat{1}).
\]
Just like in Example \ref{resiprookkimatriisi}, also in this case we are able to define a disc around $1$ that contains all the eigenvalues of $\mx M_{S,f}^{1, -1, 0, 0}$.
\end{example}

\begin{remark}
If the function $f$ is semimultiplicative, then 
\[
\left|\frac{f(x_i\wedge x_j)f(x_i\vee x_j)}{f(x_i)f(x_j)}\right|^\alpha=\left|\frac{f(x_i\wedge x_j)f(x_i\vee x_j)}{f(x_i)f(x_j)}\right|^\beta=1
\]
for all $i,j\in\{1,2,\ldots,n\}$. Thus a semimultiplicative function automatically satisfies conditions \eqref{eq:ehto2} and \eqref{eq:ehto4}.
\end{remark}

We conclude this section by considering some classical examples.

\begin{example}
Wintner \cite{Wintner} and subsequently also Linqvist and Seip \cite{lindqvist} studied the $n\times n$ matrix with
\[
\left(\frac{\gcd(i,j)}{\text{lcm}(i,j)}\right)^\alpha
\]
as its $ij$ element ($\alpha\in\Rset$). Here we have $S=\{1,2,\ldots,n\}$ and $(P,\preceq)$ may be taken to be $(\Zset_+,|)$. The set $S$ is clearly meet closed. Further we have $\beta=-\alpha$, $\gamma=\delta=0$ and $f=N$, which is trivially semimultiplicative. Thus condition \eqref{eq:ehto2} is satisfied and, with Theorem \ref{omatulos1}, all the eigenvalues of the matrix $\mx M_{S,f}^{\alpha, -\alpha, 0, 0}$ belong to the region
\[
\bigcup_{k=1}^n \left\{ z\in\Cset\ \Big|\ \abs{z-1}\leq C_n\cdot\max_{1\leq i\leq n}i^{-2\alpha}\cdot\max_{1\leq i\leq n}\abs{d_i}-1 \right\} ,
\]
where
\[
d_i=\sum_{\substack{z\mid i \\ z\nmid j\ \textrm{for}\ j<i}} (N^{2\alpha}*\mu)(z),
\]
$\mu$ is the number-theoretic Möbius function and $*$ is the Dirichlet convolution. Since the only number $z$ that satisfies $z\mid i$ and $z\nmid j$ when $j<i$ is the number $i$ itself, $d_i$ simplifies into
\[
d_i=(N^{2\alpha}*\mu)(i)=J_{2\alpha}(i),
\]
where $J_{2\alpha}$ is the generalized Jordan totient function. If $\alpha>0$, we even have $J_{2\alpha}(i)>0$ for all $i=1,\ldots,n$. As it was with the reciprocal matrices, also in this case this region is in fact a $1$-centered disc. But since $\mx M_{S,f}^{\alpha, -\alpha, 0, 0}$ is real and symmetric, all the eigenvalues are real. Therefore the disc may be constricted into a real interval with $1$ as its midpoint. Thus the eigenvalues of $\mx M_{S,f}^{\alpha, -\alpha, 0, 0}$ all belong to the interval
\[
\left\{ z\in\Rset\ \Big|\ \abs{z-1}\leq C_n\cdot\max_{1\leq i\leq n}i^{-2\alpha}\cdot\max_{1\leq i\leq n}J_{2\alpha}(i)-1 \right\}.
\]
In the special case when $\alpha=\frac{1}{2}$ we have 
\[N^{2\alpha}*\mu=N*\mu=\phi,\]
where $\phi$ is the Euler totient function. In this case the elements of $\mx D$ become
\[
d_i=\phi(i)>0.
\]
Since for all $i\geq2$ we have $\phi(i)\leq i-1$, it follows that $\max_{1\leq i\leq n}\phi(i)\leq n-1$. In addition, $\max_{1\leq i\leq n}i^{-1}=1$, and this maximum is obtained when $i=1$. Thus the eigenvalues of the matrix $\mx M_{S,f}^{\frac{1}{2}, -\frac{1}{2}, 0, 0}$ belong to the interval
\[
\left\{ z\in\Rset\ \Big|\ \abs{z-1}\leq C_n\cdot(n-1)-1 \right\}=[2-C_n\cdot(n-1),C_n\cdot(n-1)].
\]
\end{example}

\section{Estimating the constant \texorpdfstring{$c_n$}{cn}}\label{c_n}

The constant $c_n$ was originally defined by Hong and Loewy \cite{hong}, but they did not give any approximations for it. Ilmonen et al. \cite[Section 7]{IHM} easily found a relatively good upper bound
\begin{equation}\label{eq: vakio T_n}
T_n=\sqrt{(2n-1)+(2n-3)\cdot 4+(2n-5)\cdot 9+\cdots+3\cdot (n-1)^2+n^2}
\end{equation}
for their other constant $C_n$, but they did not manage to prove anything about the constant $c_n$. Instead they end up presenting the following conjecture.

\begin{conjecture}\label{konjektuuri}
Let $\mx Y_0=[(\mx Y_0)_{ij}]$, where
\[
(\mx Y_0)_{ij}=
\left\{ \begin{array}{llll}
0 & \textrm{if}\ j>i,\\
1 & \textrm{if}\ j=i,\\
0 & \textrm{if}\ i>j\ \textrm{and}\ i+j\ \textrm{is even,}\\
1 & \textrm{if}\ i>j\ \textrm{and}\ i+j\ \textrm{is odd.}\\
\end{array}\right.
\]
Then $c_n=\kappa(\mx Y_0\mx Y_0^T)$.
\end{conjecture}

Calculations have shown that this conjecture is true for $n=2,3,\ldots,7$, but generally this problem is still open and appears to be quite hard to solve. However, the next theorem shows that it is possible to obtain a lower bound for $c_n$. Unfortunately this lower bound is far from accurate and thus for the most part is only of some theoretical interest. 

\begin{theorem}\label{Mikan lemma}
The constant $c_n$ is bounded below by $\left(\frac{6}{n^4+2n^3+2n^2+n}\right)^{\frac{n-1}{2}}$. 
\end{theorem}
\begin{proof}
Let $\mx X_0\in K(n)$ be the triangular $0,1$ matrix with $c_n=\kappa(\mx X_0\mx X_0^T)$ and $\mx M_0=\mx X_0\mx X_0^T$. Let
\[
g(\lambda)=\det(\mx M_0-\lambda\mx I_n)=(-1)^n\lambda^n+a_{n-1}\lambda^{n-1}+\cdots+a_1\lambda+a_0\in\Zset[\lambda]
\]
be the characteristic polynomial of the matrix $\mx M_0$. Now
\[
g(0)=a_0=\det(\mx M_0)=\det(\mx X_0\mx X_0^T)
=\det(\mx X_0)\det(\mx X_0^T)=1^n\cdot 1^n=1,
\]
since all the diagonal elements of $\mx X_0$ are equal to $1$. Since $\mx M_0$ is clearly positive definite, let $\lambda_1,\lambda_2,\ldots,\lambda_n\in\Rset_+$ be the eigenvalues of $\mx M_0$, where
\[
0<c_n=\lambda_1\leq\lambda_2\leq\cdots\leq\lambda_n\leq C_n.
\]
Thus $g(\lambda)$ may be written as
\[
g(\lambda)=(-1)^n\cdot(\lambda-\lambda_1)(\lambda-\lambda_2)\cdots(\lambda-\lambda_n),
\]
from which we obtain
\[
1=a_0=\underbrace{\lambda_1}_{=c_n}\underbrace{\lambda_2}_{\leq C_n}\cdots\underbrace{\lambda_n}_{\leq C_n}\leq c_n(C_n)^{n-1}\leq c_nT_n^{n-1},
\]
where $T_n$ is the upper bound for $C_n$ found in \cite{IHM} and presented in \eqref{eq: vakio T_n}. By dividing this last inequality by $(T_n)^{n-1}>0$ we obtain $\left(\frac{1}{T_n}\right)^{n-1}\leq c_n$. The claim now follows by observing that $$T_n=\sqrt{\frac{1}{6}n(n+1)(n^2+n+1)}=\sqrt{\frac{1}{6}(n^4+2n^3+2n^2+n)}$$ (this can easily be proven by induction, but we omit this for the sake of brevity).
\end{proof}

If Conjecture \ref{konjektuuri} holds, then we are able to slightly improve the lower bound presented in Theorem \ref{Mikan lemma}. We only need to calculate
\begin{align*}
\mx Y_0\mx Y_0^T&=
\left[ \begin{array}{ccccccc}
1 & 0 & 0 & \cdots & 0 & 0 \\
1 & 1 & 0 & \cdots & 0 & 0 \\
0 & 1 & 1 & \cdots & 0 & 0 \\
1 & 0 & 1 & \cdots & 0 & 0 \\
0 & 1 & 0 & \  & \vdots & \vdots \\
1 & 0 & 1 & \  & \  & \  \\
\vdots & \vdots & \vdots & \cdots  & 1 & 0 \\
\vdots & \vdots & \vdots & \cdots  & 1 & 1 \\
\end{array} \right]
\left[ \begin{array}{cccccccc}
1 & 1 & 0 & 1 & 0 & 1 & \cdots & \ \\
0 & 1 & 1 & 0 & 1 & 0 & \cdots & \ \\
0 & 0 & 1 & 1 & 0 & 1 & \cdots & \ \\
\vdots & \vdots & \vdots &\vdots & \vdots & \vdots & \vdots  & \vdots \\
0 & 0 & 0 & 0 & 0 & \cdots & 1 & 1 \\
0 & 0 & 0 & 0 & 0 & \cdots & 0 & 1 \\
\end{array} \right]\\
&=\left[ \begin{array}{ccccccccc}
1 & 1 & 0 & 1 & 0 & 1 & 0 & 1 & \ddots \\
1 & 2 & 1 & 1 & 1 & 1 & 1 & 1 & \ddots \\
0 & 1 & 2 & 1 & 1 & 1 & 1 & 1 & \ddots \\
1 & 1 & 1 & 3 & 1 & 2 & 1 & 2 & \ddots \\
0 & 1 & 1 & 1 & 3 & 1 & 2 & 1 & \ddots \\
1 & 1 & 1 & 2 & 1 & 4 & 1 & 3 & \ddots \\
0 & 1 & 1 & 1 & 2 & 1 & 4 & 1 & \ddots \\
1 & 1 & 1 & 2 & 1 & 3 & 1 & 5 & \ddots \\
\ddots & \ddots & \ddots & \ddots & \ddots & \ddots & \ddots & \ddots & \ddots \\
\end{array} \right]=\mx N_0,
\end{align*}
where the last row and column vectors are equal to 
\[
\left[ \begin{array}{cccccccccccc}
1 & \underline{1} & 1 & \underline{2} & 1 & \underline{3} & \cdots & \underline{\frac{n}{2}-2} & 1 & \underline{\frac{n}{2}-1} & 1 & \frac{n}{2}+1 \\
\end{array} \right]
\]
when $n$ is even and equal to
\[
\left[ \begin{array}{ccccccccccccc}
\underline{0} & 1 & \underline{1} & 1 & \underline{2} & 1 & \underline{3} & \cdots & \underline{\frac{n-1}{2}-2} & 1 & \underline{\frac{n-1}{2}-1} & 1 & \frac{n+1}{2} \\
\end{array} \right]
\]
when $n$ is odd. Clearly 
\[\rho(\mx Y_0\mx Y_0^T)=\rho(\mx N_0)\leq\matrixnorm{\mx N_0}_F,\]
where $\matrixnorm{\mx N_0}_F$ is the Frobenius norm of the matrix $\mx N_0$. It is now a cumbersome although an elementary task to show that
\[
\matrixnorm{\mx N_0}_F=\left\{ \begin{array}{ll}
\sqrt{\frac{1}{48}(n^4+56n^2+48n)} & \textrm{if}\ n\ \textrm{is even,}\\
\sqrt{\frac{1}{48}(n^4+50n^2+48n-51)} & \textrm{if}\ n\ \textrm{is odd.}\\
\end{array}\right.
\]
Then by replacing $C_n$ with $\rho(\mx N_0)$ and $T_n$ with $\matrixnorm{\mx M_0}_F$ in the proof of \ref{Mikan lemma} we are able to prove the following result:

\begin{theorem}\label{alarajan parannus}
If Conjecture \ref{konjektuuri} holds, then $\left(\frac{48}{n^4+56n^2+48n}\right)^{\frac{n-1}{2}}$ is a lower bound for $c_n$ when $n$ is even and $\left(\frac{48}{n^4+50n^2+48n-51}\right)^{\frac{n-1}{2}}$ is a lower bound for $c_n$ when $n$ is odd.
\end{theorem}

The following Table \ref{taulukko} shows the behaviour of $c_n$ and its lower bounds for $1\leq n\leq7$.

\begin{table}[ht!]
	\centering
	\caption{Some values of the constant $c_n$ and its lower bounds.}\label{taulukko}
		\begin{tabular}{|c|c|c|c|}
			\hline
    $n$  & Lower bound  & Lower bound  & Approximate \\
    \    & by Theorem \ref{Mikan lemma} & by Theorem \ref{alarajan parannus} & value for $c_n$ \\
    \hline
    $1$ & $1$ & $1$ & $1$ \\
    $2$ & $0.377964$ & $0.377964$ & $0.381966$ \\
    $3$ & $0.0384615$ & $0.0769231$ & $0.198062$ \\
    $4$ & $0.00170747$ & $0.00674936$ & $0.0870031$ \\
    $5$ & $4.16233\cdot10^{-5}$ & $5.40833\cdot10^{-4}$ & $0.0370683$ \\
    $6$ & $6.36185\cdot10^{-7}$ & $2.05280\cdot10^{-5}$ & $0.0148276$ \\
    $7$ & $6.64148\cdot10^{-9}$ & $8.16298\cdot10^{-7}$ & $0.00581700$ \\
    \hline
		\end{tabular}
\end{table}

\noindent
\textbf{Acknowledgement} The author wishes to thank the referee for careful reading and for useful comments.

\end{document}